\documentclass{amsart}

\usepackage{amssymb,amsfonts}
\usepackage{graphicx}
\usepackage[square,sort,comma,numbers]{natbib}
\usepackage{url}
\usepackage{subcaption}
\usepackage[all,arc]{xy}
\usepackage{caption}
\usepackage{enumerate}
\usepackage{multicol}
\usepackage{mathrsfs}
\usepackage{tikz,graphicx}
\usepackage[percent]{overpic}
\newtheorem{thm}{Theorem}[section]

\newtheorem{prop}[thm]{Proposition}
\newtheorem{lem}[thm]{Lemma}

\usepackage{graphicx}
\theoremstyle{definition}

\newtheorem{exmp}[thm]{Example}

\theoremstyle{remark}
\newtheorem{rem}[thm]{Remark}

\newcommand{\ra}{\rightarrow}
\newcommand{\vb}{\text{vb}}
\makeatletter
\let\c@equation\c@thm
\makeatother
\numberwithin{equation}{section}

\title{Wirtinger numbers for virtual links}

\author{Puttipong Pongtanapaisan}

\begin{document}

\begin{abstract}
The Wirtinger number of a virtual link is the minimum number of generators of the link group over all meridional presentations in which every relation is an iterated Wirtinger relation arising in a diagram. We prove that the Wirtinger number of a virtual link equals its virtual bridge number. Since the Wirtinger number is algorithmically computable, it gives a more effective way to calculate an upper bound for the virtual bridge number from a virtual link diagram. As an application, we compute upper bounds for the virtual bridge numbers and the quandle counting invariants of virtual knots with 6 or fewer crossings. In particular, we found new examples of nontrivial virtual bridge number one knots, and by applying Satoh's Tube map to these knots we can obtain nontrivial weakly superslice links.
\end{abstract}
\maketitle
\section{Introduction}
Virtual knots were introduced by Kauffman \cite{MR1721925} as a generalization of classical knot theory, and since then many invariants have been developed to help distinguish virtual knots. One can represent virtual knots geometrically as knots in thickened surfaces up to stable equivalence. Therefore, an oriented virtual knot invariant is also an invariant of a knot in a thickened surface. Among these invariants, the virtual bridge number has been studied in \cite{MR3334661,MR2468375,MR3024023,MR2812268,MR3105303,MR3431020}. A naive way to determine an upper bound for the virtual bridge number is to consider a virtual knot diagram from a knot table, and count the number of overbridges in the diagram. However, since diagrams from knot tables are crossing number minimizing and not necessarily bridge number minimizing, one can get upper bounds that are much larger than the actual virtual bridge numbers. Previously, more accurate upper bounds were obtained by performing a sequence of extended Reidemeister moves on virtual knot diagrams to reduce the number of overbridges. Finding such a sequence of moves can be time-consuming and difficult. This motivates the search for an alternative way to obtain stronger upper bounds from the diagrams without having to perform any extended Reidemeister moves.

In \cite{blair2017wirtinger}, the authors defined the Wirtinger number of a classical link in 3-space to be the minimum number of meridional generators of the link group where all the relations in the group presentation are iterated Wirtinger relations in the link diagram and showed that it equals the bridge number of the link. This result has some beneficial consequences. First, the Wirtinger number of a link is bounded below by the meridional rank of the link group. Therefore, the main theorem of \cite{blair2017wirtinger} gave rise to an alternative approach to Cappell and Shaneson's Meridional Rank Conjecture \cite{(Ed.)95problemsin}, which asks if the bridge number of a knot equals the meridional rank of the knot group. In particular, the conjecture is true if every link admits a minimal meridional presentation in which all relations arise as iterated Wirtinger relations in a diagram. Furthermore, the Wirtinger number is algorithmically computable and gives rise to a useful combinatorial tool to obtain strong upper bounds on classical bridge numbers from knot diagrams without having to perform Reidemeister moves. This allowed the authors in \cite{blair2017wirtinger} to determine the bridge numbers of nearly half a million classical knots.

In this paper, we extend the notion of the Wirtinger number to virtual links. We show that the Wirtinger number equals the virtual bridge number. As an application, we compute upper bounds for the virtual bridge numbers of virtual knots with 6 or fewer crossings. From these upper bounds, we can obtain further information about other virtual and classical knot invariants. For instance, if $X$ is a finite quandle, then an upper bound for the virtual bridge number of a knot is also an upper bound for the number of coloring of the knot by $X$. In particular, for virtual bridge number one knots, there are only $|X|$ colorings of the knot by $X$, where $|X|$ denotes the order of $X$. Moreover, by applying Satoh's Tube map to these knots we can obtain interesting embeddings of unknotted ribbon tori in the 4-sphere, and considering cross-sections of these tori leads to diagrams of nontrivial weakly superslice link. The proofs presented here are inspired by \cite{blair2017wirtinger}.
\section{Preliminaries}

\subsection{Virtual Links} 
In this section, we recall several equivalent definitions of virtual links. The first definition is in terms of a virtual link diagram. A \textit{virtual link diagram} is an immersion of $n$ circles into the 2-sphere such that each double point is marked as either a classical crossing or a virtual crossing (see Figure 1). A \textit{virtual link} is an equivalence class of virtual link diagrams under planar isotopies and the \textit{extended Reidemeister moves} shown in Figure 2. \\

\begin{figure}[!ht]
  \centering
    \includegraphics[width=0.4\textwidth]{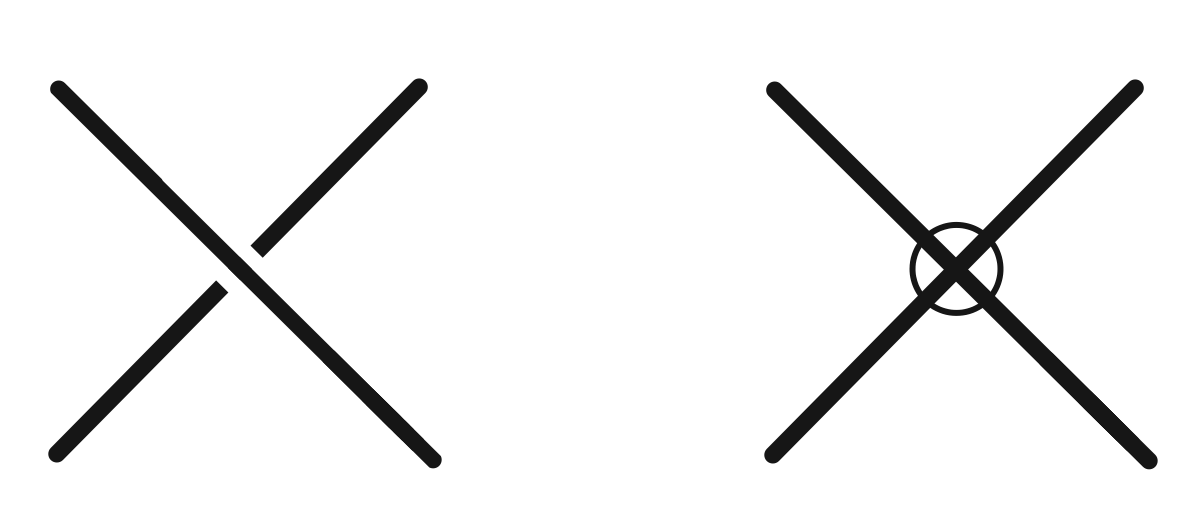}
        \caption{(Left) A classical crossing. (Right) A virtual crossing.}
\end{figure}

\begin{figure}[!ht]
  \centering
    \includegraphics[width=0.7\textwidth]{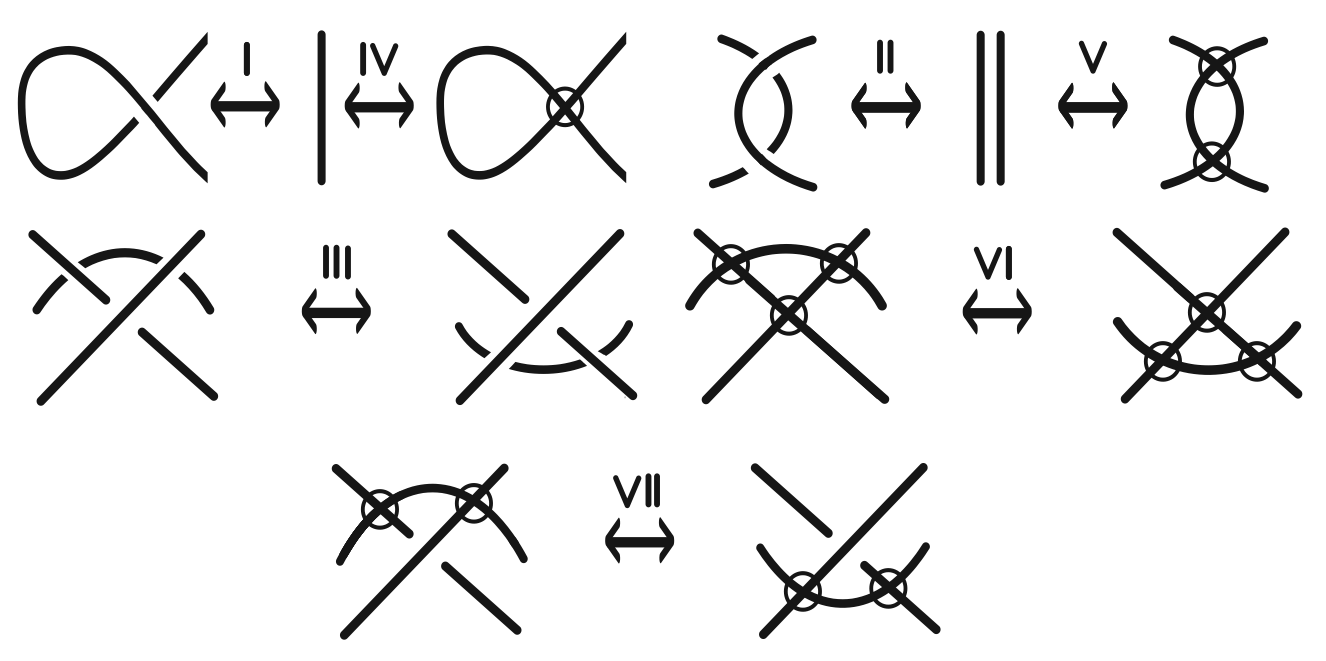}
        \caption{Extended Reidemeister moves.}
\end{figure}

A virtual link diagram can be represented as a link diagram in an oriented surface $\Sigma$ by adding handles to the sphere where the diagram is drawn to \textit{desingularize} the virtual crossings (see Figure 3). We may assume that $\Sigma$ is connected because we can take the connected sum of the components if $\Sigma$ is not connected after desingularization. 
It is shown in \cite{MR1905687} that one can regard a virtual link as a link diagram in $\Sigma$ up to Reidemeister moves on the diagram, orientation-preserving homeomorphisms of the surface, stabilizations, and destabilizations. The stabilization operation consists of removing two open disks in $\Sigma$ disjoint from the link diagram, and then joining the resulting boundary components by an annulus. The destabilization operation consists of cutting $\Sigma$ along a simple closed curve disjoint from the link diagram, and then capping off the resulting boundary components with a pair of disks. It is well-known that one can also regard a virtual link as an embedded link in thickened surfaces up to ambient isotopies, stabilizations, and destabilizations. Furthermore, Kuperberg \cite{MR1997331} showed that there exists a unique link in a thickened surface of minimum genus corresponding to each virtual link. 

\begin{figure}[!ht]
  \centering
  \includegraphics[width=0.5\textwidth]{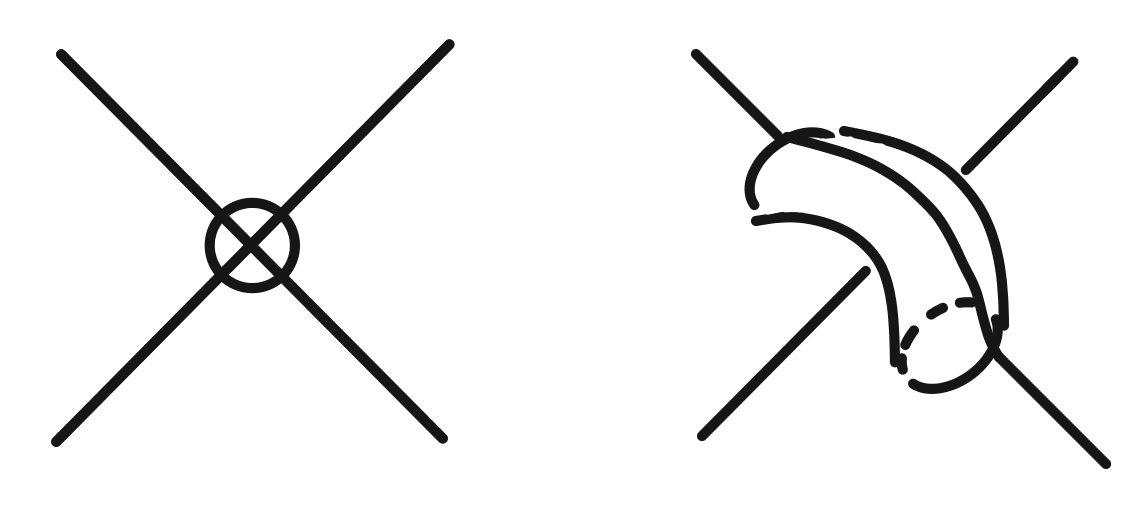}
        \caption{Desingularizing virtual crossings.}
\end{figure}

The final definition is in terms of Gauss diagrams. Given an oriented virtual link diagram $p:S^1\sqcup S^1 \sqcup \cdots \sqcup S^1\ra \mathbb{R}^2$, its \textit{Gauss diagram} $D$ is a decoration of the oriented circles in the domain of $p$ such that the pre-images of the classical crossings are connected by \textit{chords}, which are signed arrows starting from the over crossing to the under crossing. The sign of the arrow indicates the sign of a crossing using the right hand rule. The classical Reidemeister moves can be translated to moves on the Gauss diagrams. Virtual links are then in one-to-one correspondence with Gauss diagrams modulo the Reidemeister moves \cite{goussarov2000finite}. See Figure 4 for an example of a virtual link diagram, and its corresponding Gauss diagram. It is well-known that a Gauss diagram does not always represent a classical link, but every Gauss diagram corresponds to some virtual link. In a sense, Gauss diagrams give simpler representations of virtual links than virtual link diagrams since virtual crossings are not present. Therefore, we state our results mostly in terms of Gauss diagrams. 

\begin{figure}[!ht]
  \centering
    \includegraphics[width=0.8\textwidth]{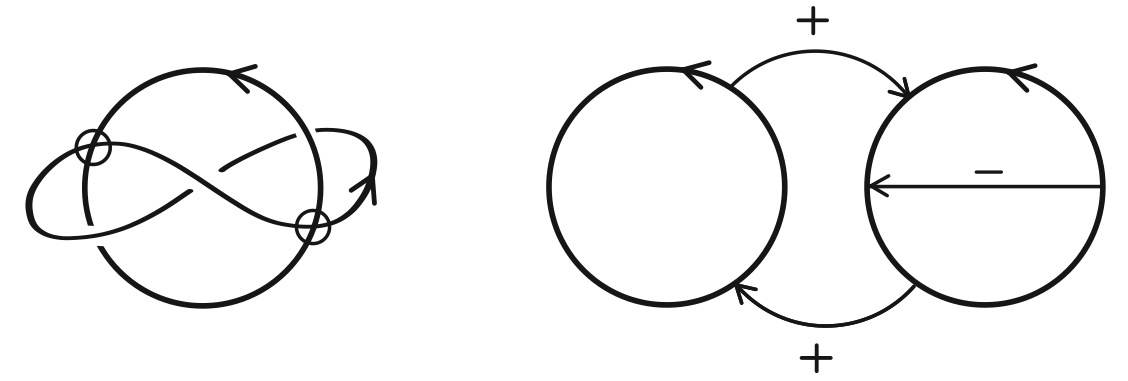}
        \caption{A virtual link diagram and its Gauss diagram.}
\end{figure}

\subsection{Virtual Bridge Number}

Let $D$ be a Gauss diagram for a virtual link. A \textit{strand} is a subarc of a circle component from one arrowhead to the next. Observe that a strand contains a finite number (possibly zero) of arrowtails, but does not contain any arrowheads. Two strands are said to be \textit{adjacent} if they are separated by an arrowhead. An \textit{overbridge} is a strand with at least one arrowtail on it. The \textit{bridge number} of $D$ is the number of overbridges of $D$, denoted vb$(D).$ If $L$ is a virtual link, then the \textit{virtual bridge number} of $L$, denoted $\vb(L)$ is the minimum bridge number taken over all Gauss diagrams $D$ of $L$. For example, the Gauss diagram in Figure 4 has two overbridges. It is a well-established fact that there is only one classical link $L$ with $\vb(L) =1$, but there are infinitely many virtual knots whose virtual bridge numbers are equal to one \cite{MR2468375}.

\begin{rem}
In this paper, we only consider Gauss diagrams where each circle component contains at least one arrowtail. If there is a circle component with no arrowtails, we can always add a trivial overbridge by performing the first Reidemeister move on the circle component. In particular, if $L$ is the $n$-component unlink, then $\vb(L) = n$. 
\end{rem}

\subsection{Link Group}

Given a Gauss diagram $D$ for a link $L$ with $n$ strands, a presentation of the link group $G_L$ is given by the following construction. The generators of $G_L$ consist of the strands of $D$. Each chord gives rise to a relation. Suppose that a circle component contains $m$ arrowheads coming from chords $c_1,c_2,...,c_m$. These $m$ chords divide the circle component into $m$ strands $a_1,...,a_m$. We order the chords $c_i$ and strands $a_i$ consistently so that the arrowhead of $c_i$ separates $a_i$ from $a_{i+1}$, modulo $n$. If the arrowtail of $c_i$ lies on the strand $b$, we impose the relation $a_{i+1}=b^{\epsilon_i}a_ib^{-\epsilon_i}$, where $\epsilon_i$ is the sign of $c_i$.
\subsection{Wirtinger Number}

We say that $D$ is \textit{$k$-partially colored} if $k$ distinct colors have been assigned to a subset of the strands of $D$. Suppose that $D_1$ is $k$-partially colored. Let $c_p$ be a chord in $D_1$ whose arrowtail lies on a colored strand $a_r$.  Suppose further that a strand $a_p$ on one side of the arrowhead of $c_p$ is colored, and the strand $a_q$ on other side of the arrowhead of $c_p$ is not colored. Then, we may extend the color on $a_p$ to $a_q$ to obtain a new $k$-partially colored diagram $D_2$. The process of extending a color in this fashion is called a \textit{coloring move}, which we denote by $D_1 \ra D_2$. Figure 5 demonstrates this process where each chord can take any signs.

For a Gauss diagram $D$ with $n$ strands, we say that $D$ is \textit{$k$-meridionally colorable} if there exists a $k$-partially colored diagram $D_0$ with $k$ colored strands $a_{i_1},...,a_{i_k}$, and a sequence of coloring moves $D_0 \ra D_1 \ra \cdots \ra D_{n-k}.$ We call the strands $a_{i_1},...,a_{i_k}$ of $D_0$ the \textit{seed strands}. The \textit{Wirtinger number} of $D$, denoted $\omega(D)$, is the minimum value of $k$ such that $D$ is $k$-meridionally colorable. Now, let $L$ be a virtual link. The \textit{Wirtinger number} of $L$, denoted $\omega(L)$, is the minimal value of $\omega(D)$ over all Gauss diagrams $D$ representing $L$.\\

\begin{figure}[!ht]
  \centering
    \includegraphics[width=0.6\textwidth]{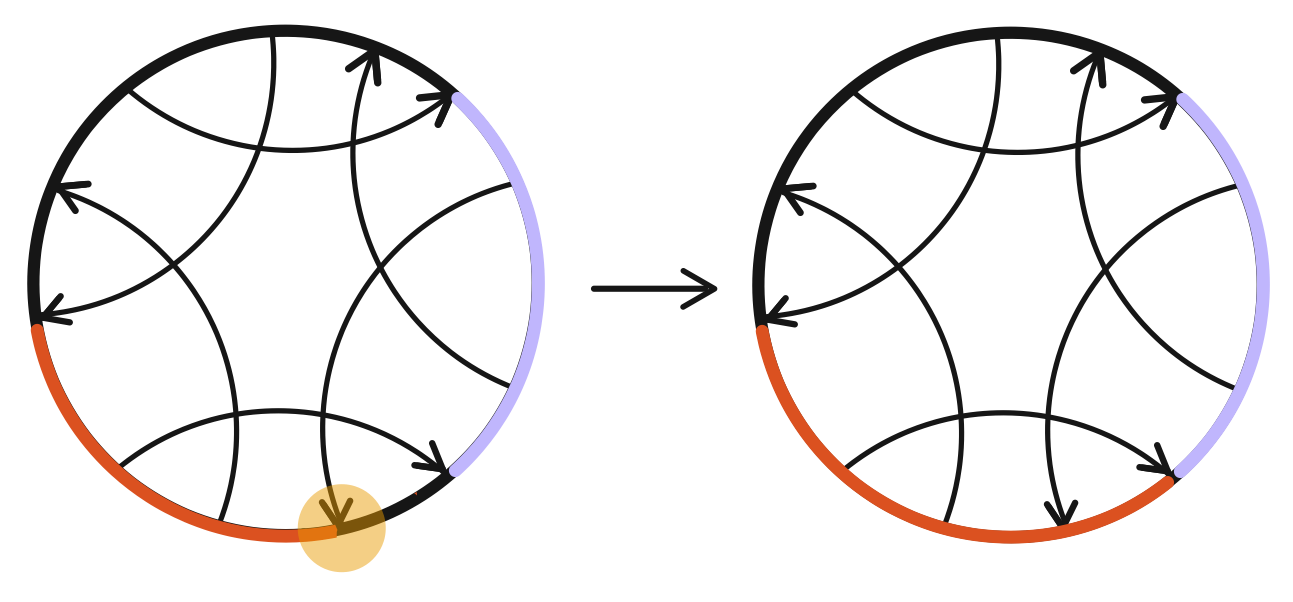}
        \caption{A coloring move.}
\end{figure}

It is useful to record the order in which strands are colored. Suppose that $D$ is $k$-meridionally colorable with seed strands $\lbrace a_{i_1},a_{i_2},...,a_{i_k}\rbrace$, and a sequence of coloring moves $D_0 \ra D_1 \ra \cdots \ra D_{n-k}.$ We associate to these coloring moves the \textit{coloring sequence} $\lbrace \alpha_j \rbrace_{j=1}^{n}$ given by $\alpha_j = a_{i_j}$ for $1 \leq j \leq k$. For, $k+1 \leq j \leq n$, we define $\alpha_j$ to be the strand that is colored in $D_{j-k}$, but not colored in $D_{j-(k+1)}$. Furthermore, given a coloring sequence $\{ \alpha_j\}$ we associate to it a \textit{height function} $h : \{ \text{{strands} of $D$} \} \ra \mathbb{R}$ by $h(\alpha_j) = \frac{1}{j+1}$. Given a set of strands $\{ a_i \}_{i=1}^n$ ordered by adjacency, we say that $h$ has a \textit{local maximum} at a strand $a_j$ if the function $h':\lbrace 1,2,...,n \rbrace \ra \mathbb{R}$ defined by $h'(i) = h(a_i)$ has a local maximum at $j$. Observe that the \textit{seed strands} $ a_{i_1},...,a_{i_k}$ generate the link group via iterated application of Wirtinger relations in $D$.

\section{Wirtinger Number and Bridge Number}
In this section, we prove that for a virtual link $L$, its Wirtinger number equals its virtual bridge number. Since a Gauss diagram represents some virtual link, we extend the results proved in \cite{blair2017wirtinger} and rephrase them in terms of Gauss diagrams. We begin by studying the case of knots.

\begin{lem}
Let $D$ be a Gauss diagram of a nontrivial virtual knot. Suppose that the arrowhead of a chord $c_p$ separates $a_p$ from $a_{q}$, and the arrowtail of $c_p$ lies on the strand $a_r$. Further, suppose that both $a_p$ and $a_q$ get assigned the same color $j$ at the end of the coloring process. If $h(a_r) \leq \min\{h(a_p),h(a_{q})\}$, then $\omega(D) = 1$, and $c_p$ is the unique chord with the property that $h(a_r) \leq \min\{h(a_p),h(a_q)\}$.
\end{lem}
\begin{proof}
Since $D$ represents a nontrivial virtual knot, $a_p \neq a_{q}$. Suppose that $h(a_p) > h(a_q) \geq h(a_r)$. Let $\delta_{q-1}$ denote the stage right before $a_q$ receives a color. Since $a_r$ is not colored at the stage $\delta_{q-1}$, this implies that $a_{q}$ must receive its color from some other strand $a_l$. Note that $a_q$ is adjacent to both $a_l$ and $a_p$. The condition that $a_{r}$ is not colored at stage $\delta_{q-1}$, and the fact that the strands of a Gauss diagram of a virtual knot lie on a circle force $a_l$ to be distinct from $a_p$. It is an easy exercise to check that at any stage of the coloring process, the set of strands of $D$ that receive the same color must be connected. Therefore, at stage $\delta_{q-1}$, the strands $a_l$ and $a_p$ will both be colored $j$. In particular, all strands of $D$ except $a_q$ are assigned the color $j$. Now, if $a_r\neq a_q$, we arrive at a contradiction because $h(a_r) \leq h(a_q)$ by assumption. Thus, $a_r$ and $a_q$ are the same strand. This implies that $D$ is 1-meridionally colorable. Furthermore, $a_q$ is the last strand in the Gauss diagram that gets colored. Thus, $c_p$ is the only chord with the property that $h(a_r) \leq \min\{h(a_p),h(a_q)\}$.
\end{proof}

We would like to have a result similar to Lemma 3.1 for virtual links as well. A Gauss diagram $D$ is called \textit{cut-split} if there exist two strands $a_p$ and $a_q$ that are adjacent at some arrowhead of $D$ such that $a_p = a_q$ or if $D$ contains a circle with no chords. For example, the Gauss diagram in Figure 4 is cut-split. The following lemma is the analog of Lemma 3.1 for Gauss diagrams that are not cut-split.

\begin{lem}
Let $D$ be a Gauss diagram of a virtual link that is not cut-split. Suppose that the arrowhead of a chord $c_p$ separates $a_p$ from $a_{q}$, and the arrowtail of $c_p$ lies on the strand $a_r$. If both $a_p$ and $a_q$ get assigned the same color $j$ at the end of the coloring process, then one of the following holds:\\
\begin{enumerate}
\item $h(a_r) > \min\{h(a_p),h(a_{q})\}$ 
\item The strands that get assigned the color $j$ at the end of the coloring process form one circle component $U$ of the Gauss diagram, and $c_p$ is the unique chord having an arrowhead on $U$ with the property that $h(a_r) \leq \min\{h(a_p),h(a_{q})\}$ 
\end{enumerate}

\end{lem}
\begin{proof}
Since $D$ is not cut-split, $a_p \neq a_{q}$. Suppose that $h(a_p) > h(a_q) \geq h(a_r)$. Let $\delta_{q-1}$ denote the stage right before $a_q$ receives a color. Since $a_r$ is not colored at the stage $\delta_{q-1}$, this implies that $a_{q}$ must receive its color from some strand $a_l$. Note that $a_q$ is adjacent to both $a_l$ and $a_p$.

If $a_p=a_l$, then the set of strands that are assigned the color $j$ form one circle component $U$ of $D$. Furthermore, there are exactly two arrowheads $c_p$ and $c_p'$ that touch $U$. By the definition of the coloring move, the strand $a_s$ that the arrowtail of $c_p'$ touches is already colored at the stage $\delta_{q-1}$. Therefore, we get the situation \textit{(2)} described in the statement of the lemma.

Suppose now that $a_p\neq a_l$. Then, there are more than two arrowheads touching $U$. As in the proof of Lemma 3.1, at the stage $\delta_{q-1}$, all strands of $U$ except $a_q$ are colored $j$ because otherwise, the strands that get assigned the color $j$ form a disconnected subset of $U$, which cannot happen. Thus, the stage $\delta_q$ is the first stage where every strand in $U$ gets assigned the color $j$. Now we show that $c_p$ is the unique chord incident to $U$ with the property that $h(a_r) \leq \min\{h(a_p),h(a_{q})\}$.

Suppose that there exists a chord $c_p'$, whose arrowhead separates $a_p'$ from $a_{q}'$, and the arrowtail of $c_p'$ lies on the strand $a_r'$. Suppose also that $a_p'$ and $a_q'$ get assigned the same color $j$ at the end of the coloring process, and $h(a_r') \leq \min\{h(a_p'),h(a_{q}')\}$. We can apply the argument in the previous paragraph and see that $a_{q}'$ must receive its color from some other strand $a_l'$, and that $\delta_{q'}$ is the first stage where every strand in $U$ gets assigned the color $j$. This implies that $\delta_q = \delta_{q'}$ and $a_q=a_{q}'$. Suppose $c_p \neq c_p'$. Now, $a_q$ is the arc that connects the arrowhead of $c_p$ and the arrowhead of $c_p'$. Then, at the stage $\delta_q$, $a_r'$ is already colored by the definition of coloring move and $a_r$ is uncolored by assumption. But on the other hand, since $h(a_r') \leq \min\{h(a_p'),h(a_{q}')\}$, $a_r'$ is not yet colored at stage $\delta_q$. This is a contradiction. Thus, $c_p = c_p'$.
\end{proof} 

Observe that if $D$ is a $k$-meridionally colorable Gauss diagram of a virtual link $L$ that is not cut-split, then $h$ attains a unique local maximum along each color at the seed strand. This fact was proved rigorously in \cite{blair2017wirtinger} for classical links, and it is not difficult to see that this fact generalizes to virtual links.

\begin{thm}
Let $L$ be a virtual link. The Wirtinger number and the virtual bridge number of $L$ are equal.

\end{thm}
\begin{proof}
Suppose that $L$ is an $N$-component virtual link with $\vb(L)=k$. Let $D$ be a Gauss diagram such that $\vb(D) = \vb(L)$. We assign $k$ distinct colors to the overbridges. We pick a point on one of the overbridges, say $b_1$, and travel along the circle until we encounter an arrowhead of the chord $c_1$. If the strand $b_2$ adjacent to $b_1$ is already colored, we do nothing. If $b_2$ is not yet colored, we can use the coloring move to extend the color from $b_1$ to $b_2$ since the arrowtail of $c_1$ is on some overbridge which has already received a color. Then, we start at a point on $b_2$, and follow the same procedure to make sure that the strand $b_3$ adjacent to $b_2$ receives a color. Continuing in this manner, we can color the whole circle component containing $b_1$. We can apply this procedure to every component of $D$ so that every strand of $D$ receives a color. This shows that the overbridges are the seed strands. Therefore, we have that $\omega(L) \leq \vb(L)$ as desired.

We establish the other inequality by induction on $N$.
First, we consider the case where $N=1.$ Suppose that a virtual knot $L$ admits a Gauss diagram $D$ with $c(D)$ chords, which is $k$-meridionally colorable. We can obtain a knot diagram $\widehat{D}$ on an oriented surface $\Sigma$ of genus $g$ from $D$. Let $I = [-1,1] \subset \mathbb{R}$. Let $f:\Sigma\times I \ra \mathbb{R}$ be the standard Morse function. We will construct a smooth embedding of $L$ in $\Sigma \times I$ with exactly $k$ maxima. To that end, for $t\in \mathbb{R}$, let $\Sigma_t = \Sigma \times \lbrace t \rbrace$ and $\Sigma_{[s,t]} = \Sigma \times [s,t]$. First, we embed the diagram $\widehat{D}$ in the level surface $\Sigma_0$. Next, we embed a copy $\widetilde{a_i}$ of each strand $a_i$ of $\widehat{D}$ in the level $\Sigma_{h(a_i)}$ in such a way that the orthogonal projection to $\Sigma_0$ maps $\widetilde{a_i}$ to $a_i$. At the moment, we have an embedding of a collection of disconnected line segments in $\Sigma \times I$. To obtain the knot $K$, we will construct arcs $a_{ij}$ connecting $\widetilde{a_i}$ to $\widetilde{a_j}$ for adjacent strands $a_i$ and $a_j$ in $\widehat{D}$.

Since $a_i$ is adjacent to $a_j$, they are the under-strands of some crossing $c_{ij}$ in $\widehat{D}$ with the over-strand $a_k$. For a small enough $\epsilon > 0$, let $B_{\epsilon}(c_{ij})$ be an open disk centered at $c_{ij}$ with $B_{\epsilon}(c_{ij}) \cap \widehat{D} = \lbrace a_i,a_j, a_k \rbrace.$ Consider the cylinder $B_{\epsilon}(c_{ij}) \times [0,1]$ transverse to the level surfaces of $\Sigma \times I$. It follows that $(B_{\epsilon}(c_{ij}) \times [0,1]) \cap \{\widetilde{a_1},...,\widetilde{a_{c(D)}} \} = \{\widetilde{a_i},\widetilde{a_j},\widetilde{a_k}\}$ for a small enough $\epsilon > 0$. To construct the arc $a_{ij}$, we need to consider two cases:\\

\textbf{Case I:} $L$ is a virtual knot with $\omega(L) > 1$.\\

\textbf{Case II:} $L$ is a virtual knot with $\omega(L) = 1$.\\

\textbf{Case I:} There are two subcases:\\

Subcase I: Suppose that $a_i$ and $a_j$ get assigned the same color $\mu$. Since $\omega(D) \neq 1$, Lemma 3.1 implies that $h(a_k) > \min\lbrace h(a_i),h(a_j)\rbrace$. Then $a_{ij}$ can be chosen so that the orthogonal projection of $\widetilde{a_k} \cup (\widetilde{a_i} \cup a_{ij} \cup \widetilde{a_j})$ to the level $\Sigma_0$ is the subset of $D$. (see Figure 6).

Subcase II: Suppose that $a_i$ and $a_j$ get assigned distinct colors. Let $x_{ij}$ be a point in $(B_{\epsilon}(c_{ij}) \times [0,1]) \cap \Sigma_{1/(c(D)+2)}$ so that when we orthogonally project $x_{ij}$ to the plane $\Sigma_0$, $x_{ij}$ gets mapped to the crossing $c_{ij}$. Then, we construct $a_{ij} \subset B_{\epsilon}(c_{ij}) \times [0,1]$ as the union of a smooth arc connecting $\widetilde{a_i}$ to $x_{ij}$ and another smooth arc connecting $x_{ij}$ to $\widetilde{a_j}$. As $h(a_k) \geq 1/(c(D)+1) > 1/(c(D)+2)$, $a_{ij}$ can be chosen so that the orthogonal projection of $\widetilde{a_k} \cup (\widetilde{a_i} \cup a_{ij} \cup \widetilde{a_j})$ to the level $\Sigma_0$ is the subset of $\widehat{D}$ (see Figure 6).\\

\begin{figure}
    \centering
    \begin{subfigure}[b]{0.475\textwidth}
        \includegraphics[width=\textwidth]{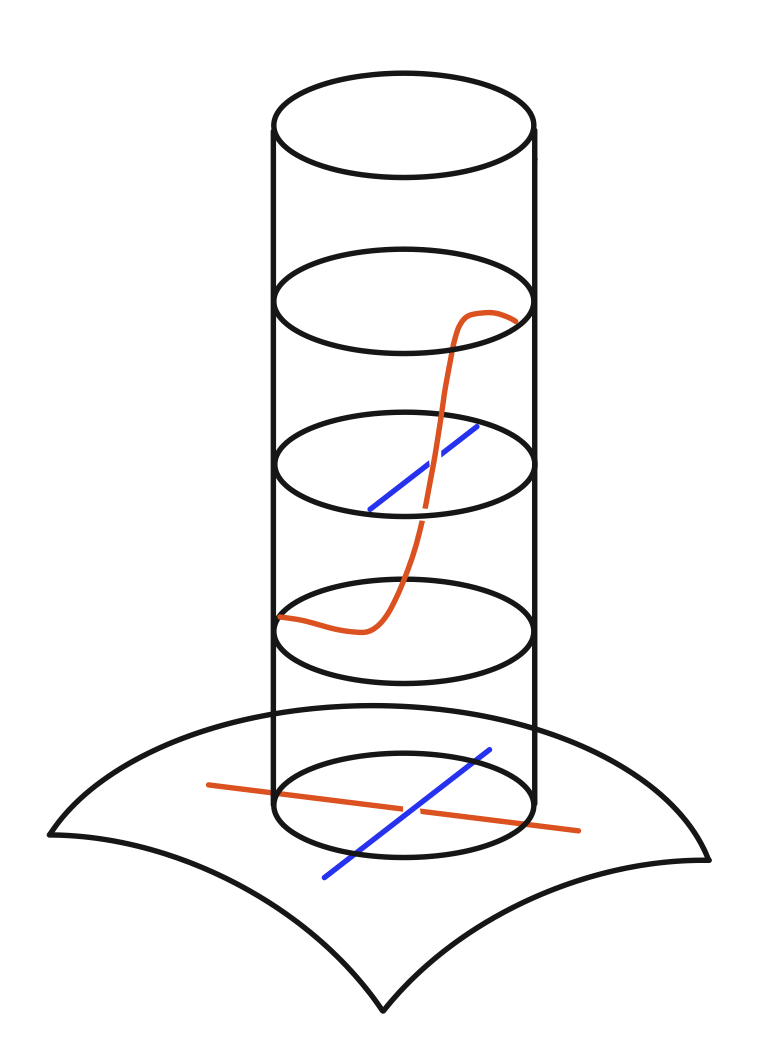}
        
    \end{subfigure}
    \begin{subfigure}[b]{0.5\textwidth}
        \includegraphics[width=\textwidth]{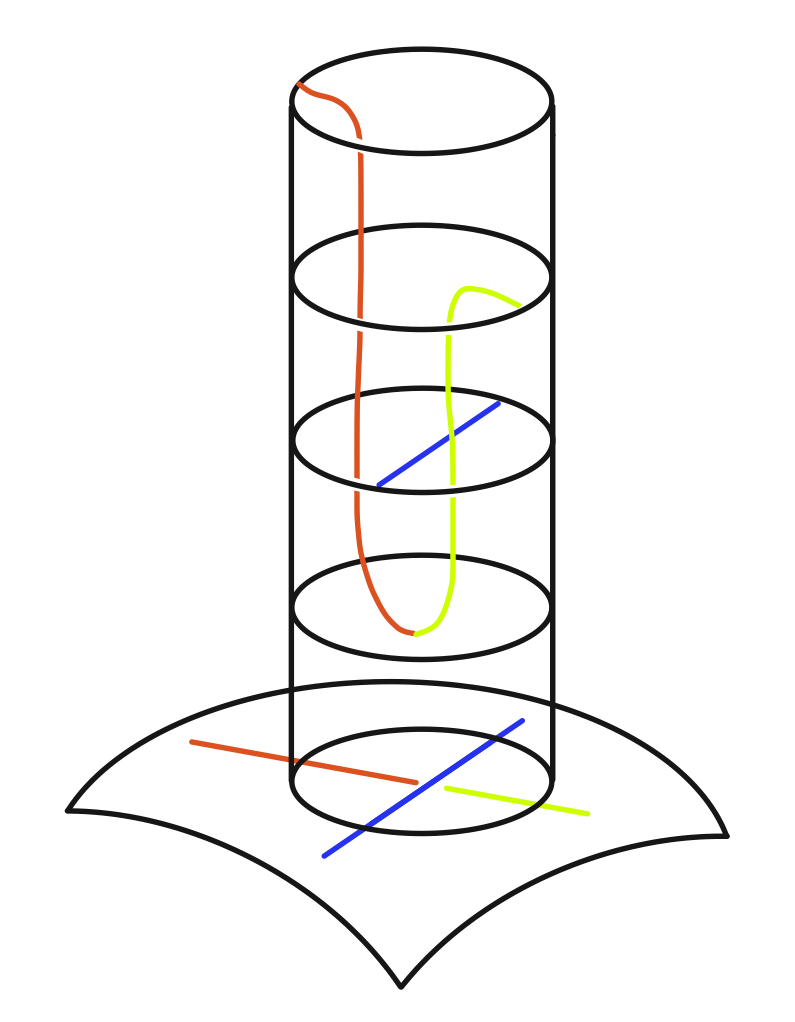}
     
    \end{subfigure}
    \caption{At left, the construction of $a_{ij}$ in Subcase I. At right, The construction of $a_{ij}$ in Subcase II.}
\end{figure}

Once we construct $a_{ij}$ for each crossing, we obtain an embedding $\widetilde{L}$ of $L$ in $\Sigma \times I$. To see that $\widetilde{L}$ has exactly $k$ local maxima, we perturb the knot slightly. Suppose that $a_i$ is adjacent to $a_{i-1}$ and $a_{i+1}$. Let $\widetilde{c_{ij}}$ be the point in $a_{ij}$ that orthogonally projects to $c_{ij}$ in $\Sigma_0$. If $a_i$ is a seed strand, we perturb the subarc $a_{(i-1)i} \cup \widetilde{a_i} \cup a_{i(i+1)}$ from $\widetilde{c_{(i-1)i}}$ to $\widetilde{c_{i(i+1)}}$ to obtain a smooth arc that monotonically increases to the midpoint of $\widetilde{a_i}$ and monotonically decreases from there. 

On the other hand, if $a_i$ is not a seed strand, then either $a_i$ has the same color as $a_{i-1}$ and $h(a_i) < h(a_{i-1})$, or it has the same color as $a_{i+1}$ and $h(a_i) < h(a_{i+1})$. As $\omega(D) > 1$, it follows from the connectedness of the set of strands having the same color that if $a_{i-1}$ and $a_{i+1}$ have the same color, then $h(a_{i+1}) < h(a_i) < h(a_{i-1})$ or vice versa. This allows us to isotope the knot in the following way regardless of the ways $a_{i+1}$ and $a_{i-1}$ are colored: we perturb the subarc $a_{(i-1)i} \cup \widetilde{a_i} \cup a_{i(i+1)}$ from $\widetilde{c_{(i-1)i}}$ to $\widetilde{c_{i(i+1)}}$ to obtain a smooth arc that is strictly increasing if $h(a_{i-1}) < h(a_{i+1})$ or strictly decreasing if $h(a_{i-1}) > h(a_{i+1})$. \\

\textbf{Case II:} Since $\omega(D) = 1$, we may not have the property that $h(a_k) > \min\lbrace h(a_i),h(a_j)\rbrace$ for all crossings on the knot diagram. But by Lemma 3.1, there is only one crossing $c_{ij}$ on the knot diagram with the property that $h(a_k) \leq \min\lbrace h(a_i),h(a_j)\rbrace$. This means that at every crossing except $c_{ij}$, we can construct $a_{ij}$ in the same way as in Subcase I of Case I. Now, let $x_{ij}$ be a point in $(B_{\epsilon}(c_{ij}) \times [0,1]) \cap \Sigma_{1/(c(D)+2)}$ so that when we orthogonally project $x_{ij}$ to the plane $\Sigma_0$, $x_{ij}$ gets mapped to the crossing $c_{ij}$. We construct $a_{ij}$ as the union of two smooth, monotonic arcs, connecting $x_{ij}$ to endpoints of $\widetilde{a_i}$ and $\widetilde{a_j}$. These two monotonic arcs can be chosen so that the orthogonal projection of $\widetilde{a_k} \cup (\widetilde{a_i} \cup a_{ij} \cup \widetilde{a_j})$ to the level $\Sigma_0$ is a subset of $\widehat{D}$. Then, the arc $a_{ij}$ contains the unique local minimum of the constructed embedding, and we obtain an embedding $\widetilde{L}$ with one maxima.\\

It follows that $\widetilde{L}$ has a projection onto $\Sigma_0$ with $k$ overbridges. Therefore, the Gauss diagram corresponding to the projection has $k$ overbridges, and $\vb(D) = k$.\\

Suppose now that $N > 1$. Suppose that $\omega(L') = \vb(L')$ for all links $L'$ of fewer than $N$ components. Let $D$ be a Wirtinger number minimizing Gauss diagram for $L$. We consider two cases:\\

\textbf{Case A:} $D$ is not cut-split.\\

\textbf{Case B:} $D$ is cut-split.\\

\textbf{Case A:} We will construct a smooth embedding of $L$ in $\Sigma \times I$ with exactly $k$ maxima. There are three subcases:\\

Subcase I: If $a_i$ and $a_j$ get assigned the same color, and $h(a_k) > \min \lbrace h(a_i),h(a_j) \rbrace$, then we follow the construction of $a_{ij}$ in Subcase I of Case I in the virtual knot case.\\

Subcase II: If $a_i$ and $a_j$ get assigned distinct colors, then we follow the construction of $a_{ij}$ in Subcase II of Case I in the virtual knot case.\\

Subcase III: If $a_i$ and $a_j$ get assigned the same color, say $\mu$, and $h(a_k) \leq \min \lbrace h(a_i),h(a_j) \rbrace$, then by Lemma 3.2, the strands that get assigned the color $\mu$ form a circle component of $D$, and $c_{ij}$ is the unique crossing with the property $h(a_k) \leq \min \lbrace h(a_i),h(a_j) \rbrace$. We construct $a_{ij}$ as in Subcase II of Case I of the knot case. Namely, we let $x_{ij}$ be a point in $(B_{\epsilon}(c_{ij}) \times [0,1]) \cap \Sigma_{1/(c(D)+2)}$ so that when we orthogonally project $x_{ij}$ to the plane $\Sigma_0$, $x_{ij}$ gets mapped to the crossing $c_{ij}$. We construct $a_{ij}$ as the union of two smooth, monotonic arcs, connecting $x_{ij}$ to endpoints of $\widetilde{a_i}$ and $\widetilde{a_j}$. These two monotonic arcs can be chosen so that the orthogonal projection of $\widetilde{a_k} \cup (\widetilde{a_i} \cup a_{ij} \cup \widetilde{a_j})$ to the level $\Sigma_0$ is a subset of $\widehat{D}$. Then, the arc $a_{ij}$ contains the unique local minimum in the color $\mu$ of the constructed embedding. \\

After we perform the construction above at every crossing, we obtain a smooth embedding of $L$ in $\Sigma \times I$. Furthermore, the standard height function $f:\Sigma \times I \ra \mathbb{R}$ restricts to a Morse function on $L$ with exactly $k$ minima and $k$ minima. This implies that $\widetilde{L}$ has a projection onto $\Sigma_0$ with $k$ overbridges. Therefore, the Gauss diagram corresponding to the projection has $k$ overbridges, and $\vb(D) = k$.\\

\textbf{Case B:} Let $D$ be the Wirtinger number minimizing Gauss digram for $L$ that is cut-split. That is, there exist two strands $a_p$, and $a_q$ that are adjacent at some arrowhead of $D$ such that $a_p = a_q$ or if $D$ contains a circle with no chords. Let $U$ be a component of $D$ that contains $a_p=a_q$. Observe that $\omega(D\backslash U) = \omega(D) - 1$ because $U$ cannot arise as a result of a coloring move. Also, $\vb(D\backslash U) = \vb(D) - 1$ because $U$ has one overbridge. Now, by the induction hypothesis, $\omega(D) - 1 =\omega(D \backslash U) = \vb(D \backslash U) = \vb(D)-1$ Thus, $\omega(D) = \vb(D)$. Since $D$ is Wirtinger minimizing, $\omega(L) = \vb(L).$ This completes the proof of Theorem 3.3.\\

\end{proof}

\section{Applications}

In this section, we present some applications of Theorem 3.3
\subsection{Computations of the Virtual Bridge Numbers}

\begin{exmp}
We will demonstrate the procedure in the proof of Theorem 3.3 with a specific example. For integers $a$ and $b$ with $a \geq 1$, $b\geq 2$, and $b$ even, Satoh and Tomiyama gave an example of a family of virtual knots $K_2(a,b)$ (see Figure 7) whose real crossing numbers equal to $a+b$ \cite{MR2833547}.

\begin{figure}[!ht]
  \includegraphics[width=1\textwidth]{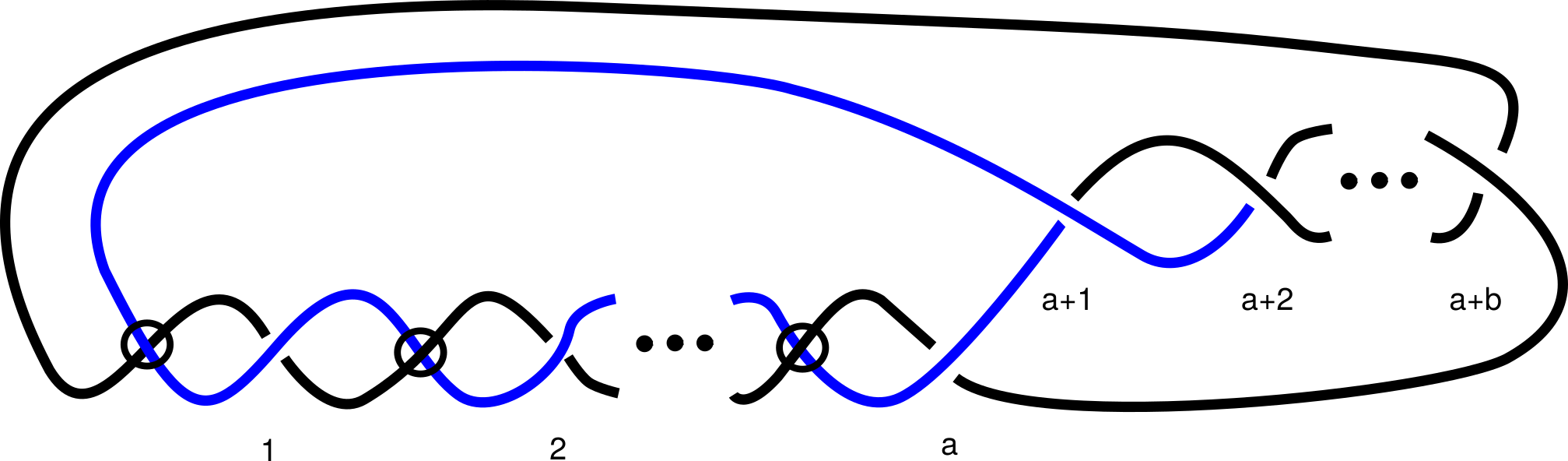}
  \caption{Satoh and Tomiyama's example}
\end{figure}

\begin{figure}
    \centering
    \begin{subfigure}[b]{0.49\textwidth}
        \includegraphics[width=\textwidth]{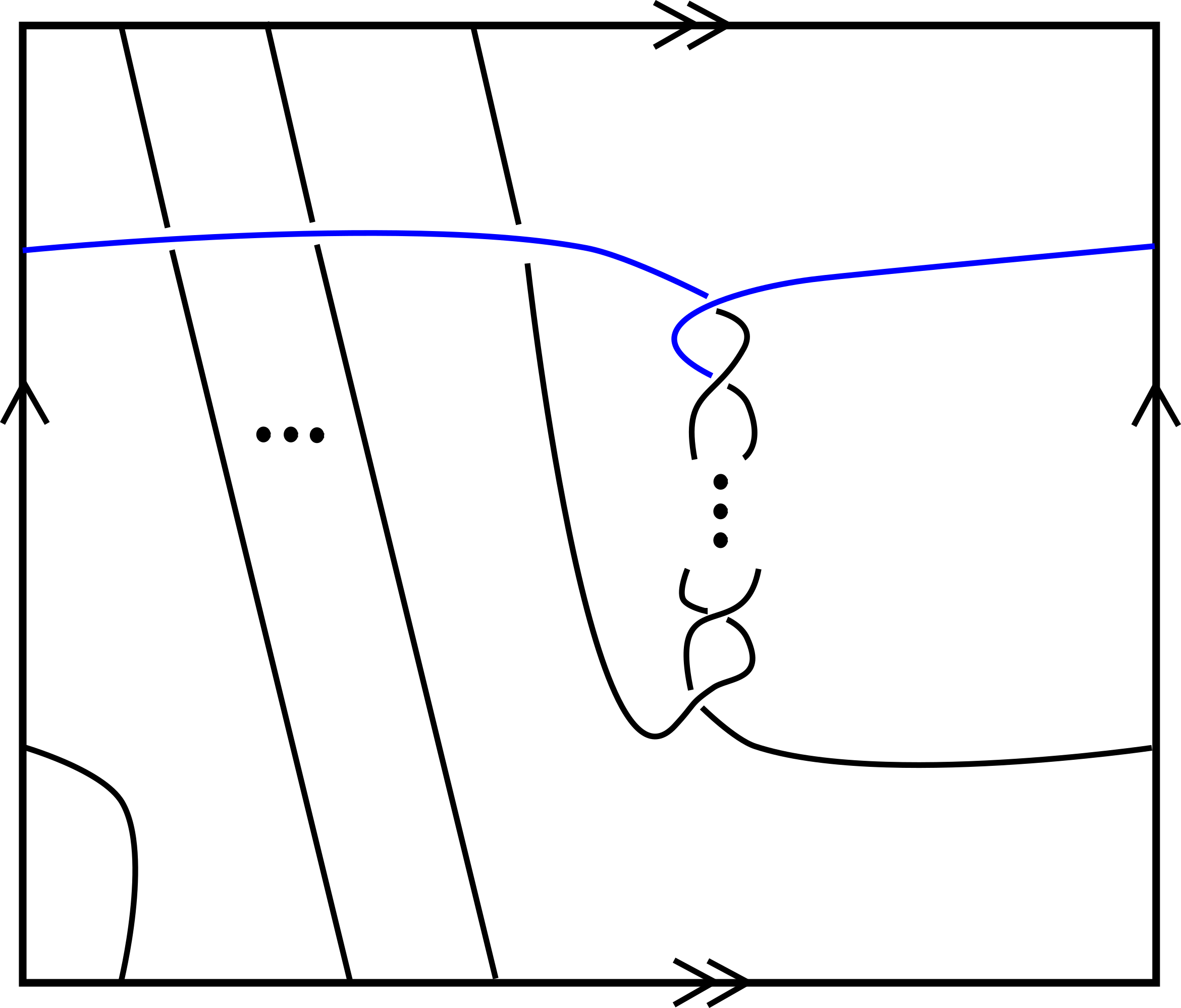}
        
    \end{subfigure}
    \begin{subfigure}[b]{0.49\textwidth}
        \includegraphics[width=\textwidth]{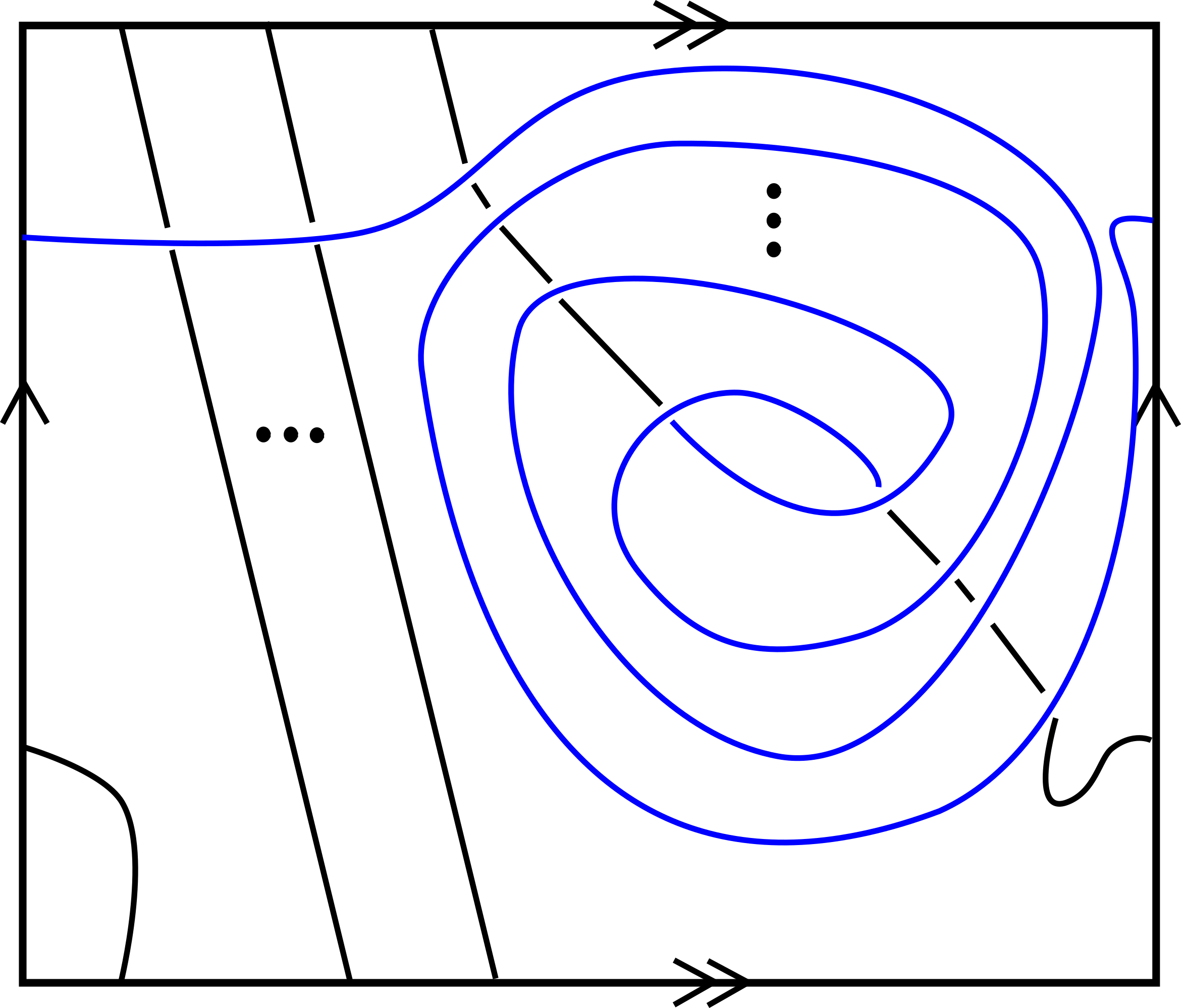}
     
    \end{subfigure}
    \caption{}
\end{figure}

Observe that the diagram in Figure 7 has $b$ overbridges. But since the blue strand is a seed strand, we can obtain another diagram of $K_2(a,b)$ with a unique overbridge as follows. We start by $K_2(a,b)$ as a knot diagram $D_2(a,b)$ on a torus $T$. This is demonstrated on the left of Figure 8. Thinking of $K_2(a,b)$ as a knot in $T \times [-1,1]$, we can start by embedding a copy of the blue seed strand in $T \times \{1/2\}$. We then embed copies of the remaining strands of $D_2(a,b)$ on different level surfaces dictated by the coloring move. At the end of the embedding process, we have copies of disconnected arcs, whose orthogonal projections to $T \times \{ 0 \}$ is $D_2(a,b)$. During the coloring process, there is no crossing where locally the overstrand gets assigned a color last. Therefore, we can connect the disjoint copies of strands lying above $T \times \{ 0 \}$ to form a knot in such a way that no addition critical points with respect to the standard Morse function $f:T \times [-1,1] \ra \mathbb{R}$ restricted to $K_2(a,b)$ are created. After a slight perturbation, we see that $f_{K_2(a,b)}$ has a unique maximum. A shadow of the bridge disk corresponding to such a maximum is drawn on the right of Figure 8 in blue.

\end{exmp}

In \cite{MR3334661}, Boden and Gaudreau showed how to use other virtual invariants to compute lower bounds for the bridge numbers. We summarize some of their results here. Suppose that the knot group $G_K$ of a virtual knot $K$ has a presentation with $n$ generators. Then, one can form the Alexander matrix $A$ associated to $G_K$, whose $(i,j)$ entry is the Laurent polynomial obtained from taking the Fox derivative of the $i$-th relation arising in the presentation of $G_K$ with respect to the $j$-th generator and substituting $t$ for each generator. The $k$-th elementary ideal $E_k$ of $G_k$ is the ideal generated by all the $(n-k) \times (n-k)$ minors of $A$. It follows that if $K$ has meridional rank $k$, then, $E_k = (1)$. Using this fact, we can bound the meridional rank and hence, the bridge number of $K$ from below.

\begin{prop}[\cite{MR3334661}]
If $K$ is a virtual knot whose $k$-th elementary ideal is proper and nontrivial, then the knot group $G_K$ has meridional rank at least $k+1$ and $K$ has bridge number at least $k+1$.
\end{prop}

Another lower bound for vb$(K)$ considered by Boden and Gaudreau comes from the Gaussian parity, which is defined in terms of Gauss diagrams as follows. Let $\mathcal{C}(D)$ denote the set of chords in a Gauss diagram $D$ and take $c \in \mathcal{C}(D)$. Then, the \textit{Gaussian parity} is a function $f:\mathcal{C}(D) \ra \mathbb{Z}_2$ where $f(c)$ is the number of elements in $\mathcal{C}(D)$ that intersects $c$ mod 2. Now, given a Gauss diagram $D$ we define its \textit{projection} $P_f(D)$ to be a Gauss diagram obtained from $D$ be erasing all chords $c$ in $\mathcal{C}(D)$ such that $f(c) = 1$. By considering the behavior of $f$ under Reidemeister moves, one can check that if $D_1$ and $D_2$ are equivalent Gauss diagrams, then $P_f(D_1)$ will be equivalent to $P_f(D_2).$ Since $P_f(D)$ is obtain from $D$ by erasing some chords, vb($K$) is bounded below by vb($P_f(K))$. We can now combine these techniques with Theorem 3.3 to determine bridge numbers of more virtual knots.

\begin{exmp}
Figure 9 shows a virtual knot $K$ together with its projection. In \cite{MR3334661}, the authors used $K$ as an example of a knot whose upper and lower quandles are trivial, but has vb$(K)> 1$. Combining with Theorem 3.3, we can conclude that vb$(K)$ = 2. More specifically, one can verify that the first elementary ideal of the knot group of $P_f(K)$ is $E_1 = (t+1,3)$. Hence, vb($K$) $\geq$ vb($P_f(K)) = 2$. On the other hand, the green strand and the blue strand in the Gauss diagram in Figure 9 on the left are seed strands. Therefore, we have that vb($K) \leq 2$. 
\end{exmp}

\begin{figure}
    \centering
    \begin{subfigure}[b]{0.35\textwidth}
        \includegraphics[width=\textwidth]{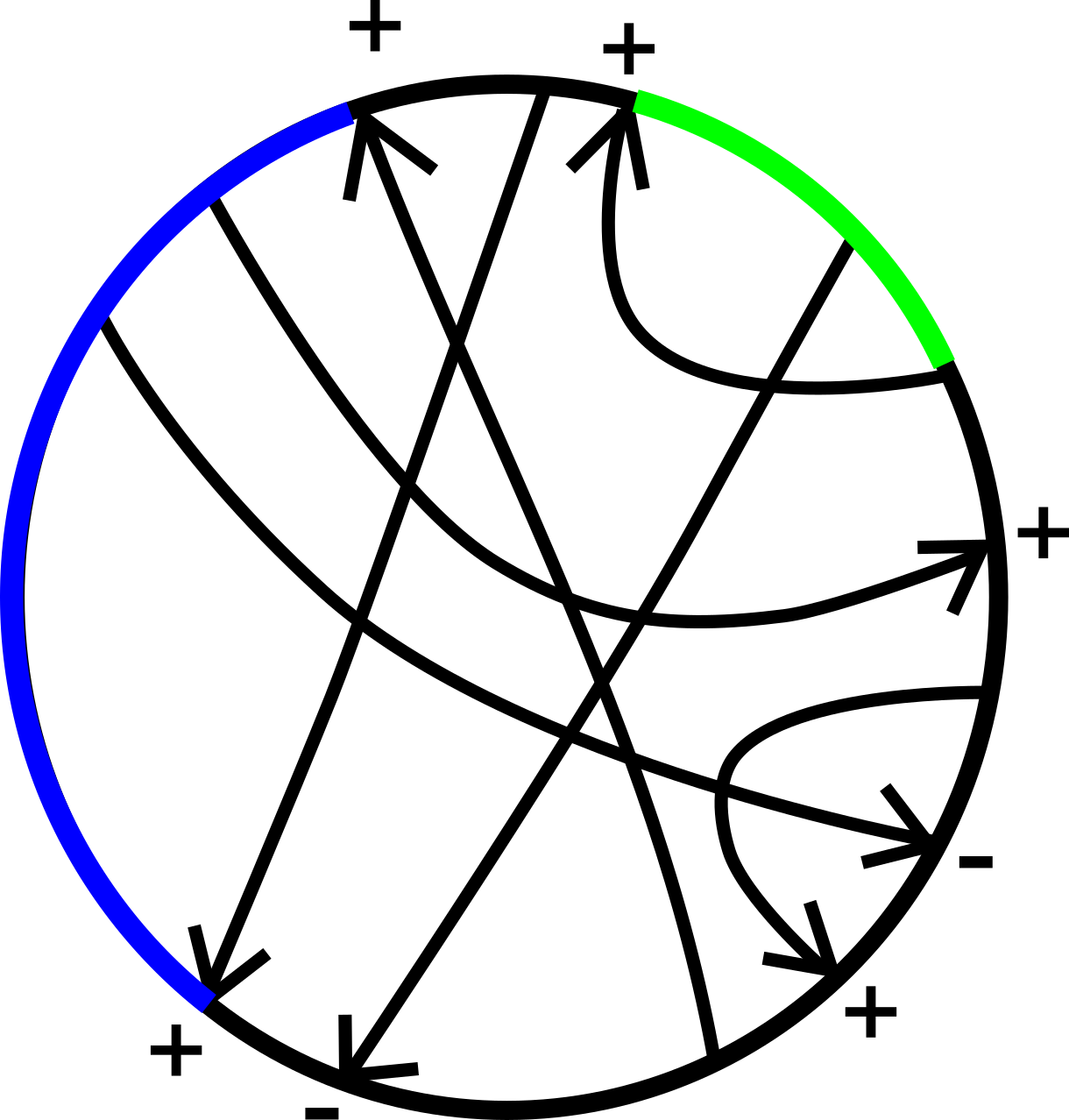}
         \end{subfigure}
    \begin{subfigure}[b]{0.34\textwidth}
        \includegraphics[width=\textwidth]{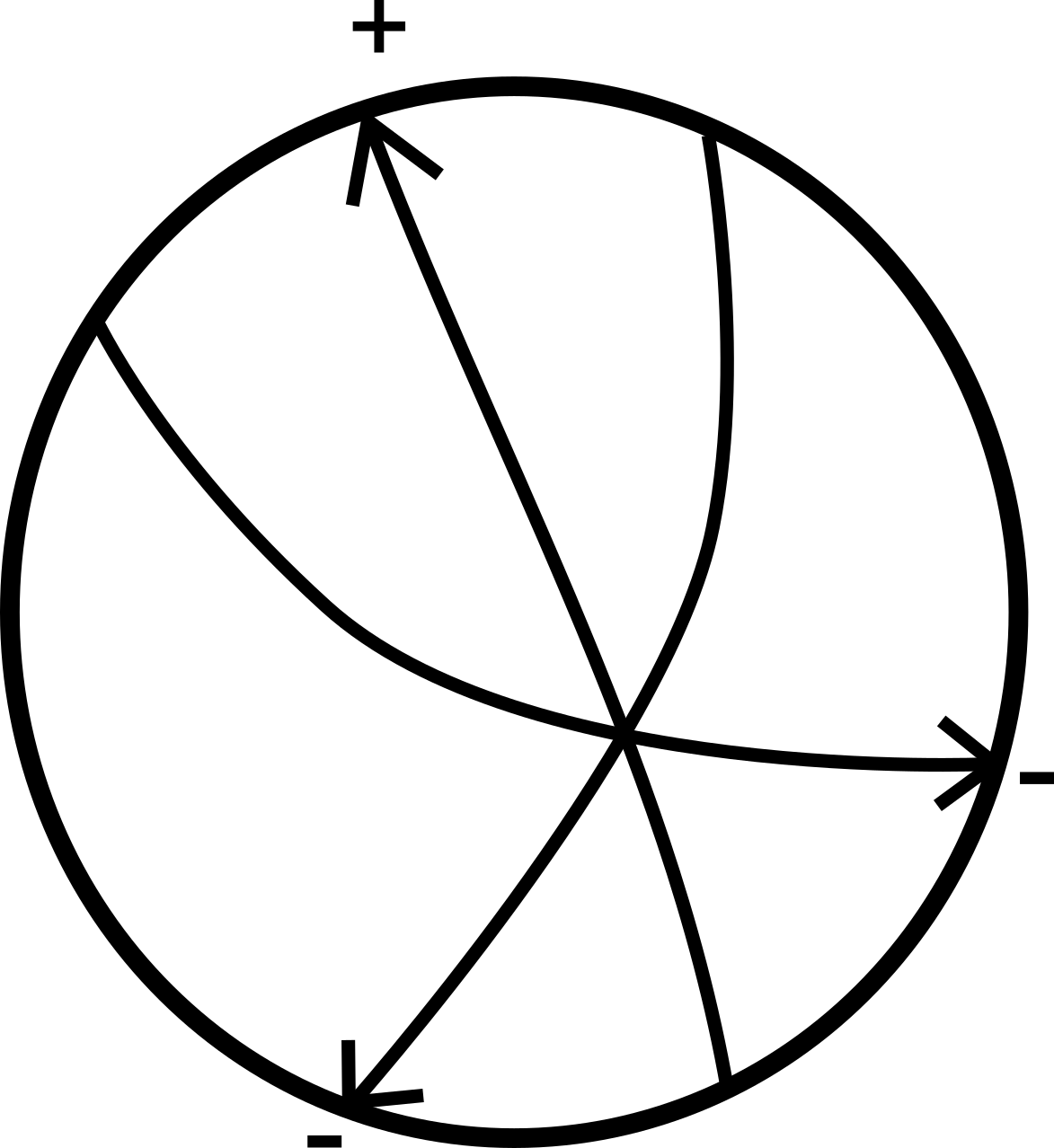}
     \end{subfigure}
    \caption{}
\end{figure}

\newpage

\begin{exmp}
The authors in \cite{blair2017wirtinger} wrote a program to calculate $\omega(D)$ for a Gauss diagram $D$ representing a classical knot. The original code is available at \cite{pv}. Starting with $k=2$, the program runs across all subsets of size $k$ of the set of strands $s(D)$ and determines whether $D$ is $k$-meridionally colorable. If not, the program repeats the process with all subsets of size $k+1$ of $s(D)$. The algorithm terminates once the first valid coloring is found. The program can be used to calculate $\omega(D)$ for a Gauss diagram $D$ representing a virtual knot if we modify the program to start at $k=1$. This allows us to compute $\omega(D)$ for all Gauss diagrams of virtual knots up to 6 crossings from Jeremy Green's virtual knot table. The spreadsheets containing the results are available at \cite{pp}. 

From this table of data, we can also get some information about the quandle counting invariants of the knots. More precisely, for a finite quandle $X$, a \textit{quandle coloring} of a knot diagram $D$ is an assignment of elements of $X$ to the strands of $D$ such that the quandle relation is satisfied at each crossing. The quandle counting invariant of a virtual knot $K$, denoted $Col_X(K)$, is the number of quandle colorings of $K$. If $D$ is $k$-meridionally colorable, then we know that $k$ strands generate the coloring of the whole diagram. So if $|X|$ denotes the order of the quandle, then there are $|X|$ possible choices of elements to assign to each seed strand. Thus, there are $|X|^k$ ways of coloring the whole diagram if we start with the seeds and generate the coloring by a sequence of coloring moves. But some of these colorings may not be quandle colorings. Since there are always $|X|$ trivial quandle colorings of a virtual knot, it follows that $|X| \leq Col_X(K) \leq |X|^k$. This implies that, for instance, virtual bridge number one knots only admit trivial quandle colorings.

\end{exmp}
\subsection{Weakly Superslice Links}

Let $D$ and $D'$ be virtual knot diagrams. We say that $D$ is \textit{welded equivalent} to $D'$ if one can be obtained from the other by a sequence of extended Reidemeister moves, planar isotopies, and welded moves (See Figure 10).

\begin{figure}[!ht]
  \centering
    \includegraphics[width=0.5\textwidth]{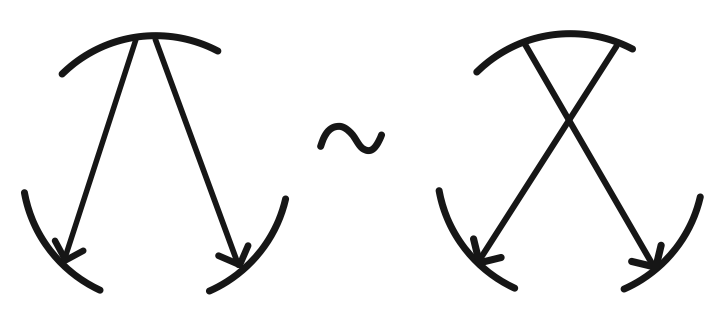}
        \caption{Welded move.}
\end{figure}

In \cite{MR1758871}, Satoh proved that any ribbon torus in $\mathbb{R}^4$ can be represented by a virtual knot diagram through the correspondence in Figure 11. We denote the ribbon torus presented by $K$ as $Tube(K)$. Furthermore, Satoh showed that if $K$ is welded equivalent to $K'$, then the corresponding ribbon tori are ambient isotopic. Using Satoh's correspondence, virtual bridge number one knots correspond to a particularly simple ribbon torus. 

\begin{figure}
    \centering
    \begin{subfigure}[b]{0.35\textwidth}
        \includegraphics[width=\textwidth]{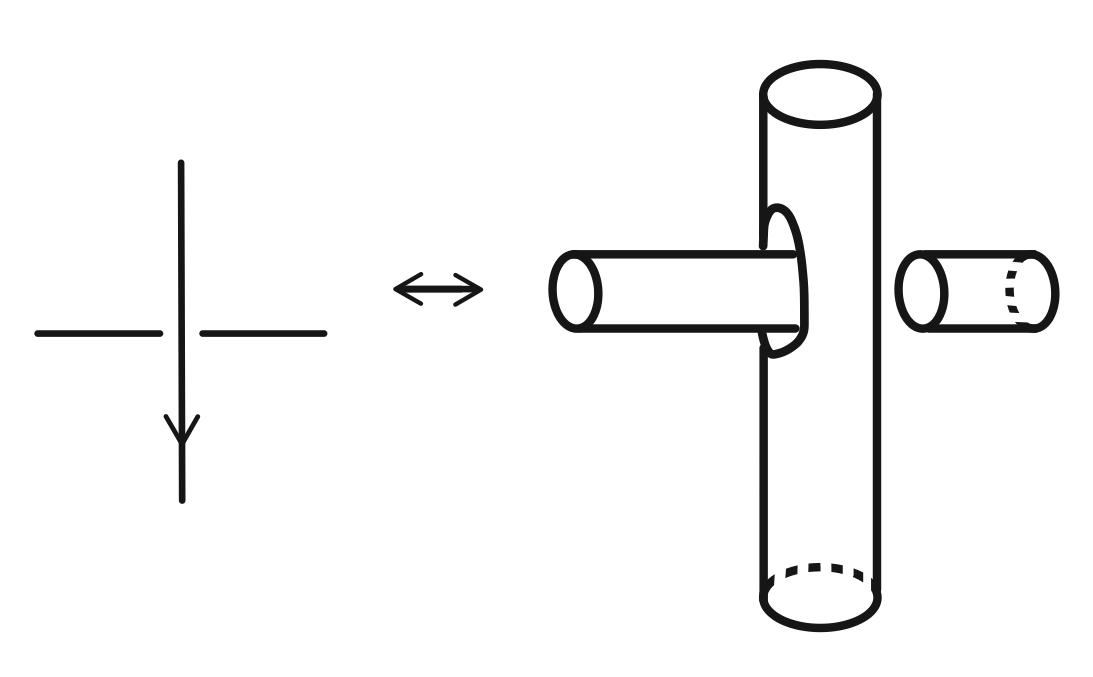}
         \end{subfigure}
    \begin{subfigure}[b]{0.34\textwidth}
        \includegraphics[width=\textwidth]{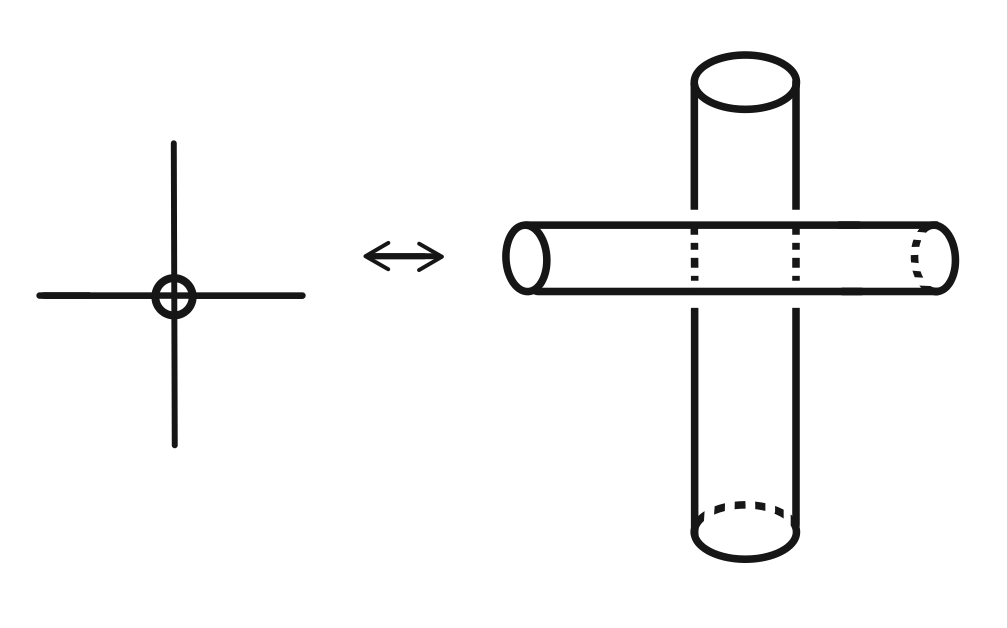}
     \end{subfigure}
    \caption{}
\end{figure}

\begin{prop}
Let $K$ be a virtual bridge number one knot. Then, Tube(K) is unknotted.
\end{prop}
\begin{proof}
Since vb$(K) = 1,$ there exists a Gauss diagram for $K$ with one overbridge, which is a sequence of arrowtails without any arrowheads. We can then repeatedly apply the welded move to unhook every pair of crossed arrowtails one by one to obtain a Gauss diagram with only parallel chords. By a repeated applications of the Reidemeister I move, we obtain the empty Gauss diagram for the unknot. This means that $K$ is welded equivalent to the unknot, and by Satoh's result, $Tube(K)$ is ambient isotopic to the unknotted torus.
\end{proof}

A $\mu$-component link $L$ in $S^3$ is said to be \textit{weakly superslice} if it is bounds a smooth planar surface properly embedded $F$ in $B^4$ such that the double of $F$ along $L$ produces an unknotted surface of genus $\mu-1$ in $S^4$. Now, from our table of upper bounds for virtual bridge numbers, we can select a knot $K$ with vb$(K) = 1$. By Proposition 4.5, $Tube(K)$ is unknotted, and interesting nontrivial links can arise as its cross-sections. 
\begin{exmp}
Consider the virtual bridge number one knot $K$ with Gauss code O1-O2-O3-U1-O4-U3-O5-U6-U2-U5-U4-O6-. The equatorial cross-section $L$ of $Tube(K)$ is depicted in Figure 12. SnapPy identified $L$ as L13n2916 from the link table. It follows by Proposition 4.5, that L13n2916 is weakly superslice.

\begin{figure}[!ht]
  \centering
    \includegraphics[width=0.6\textwidth]{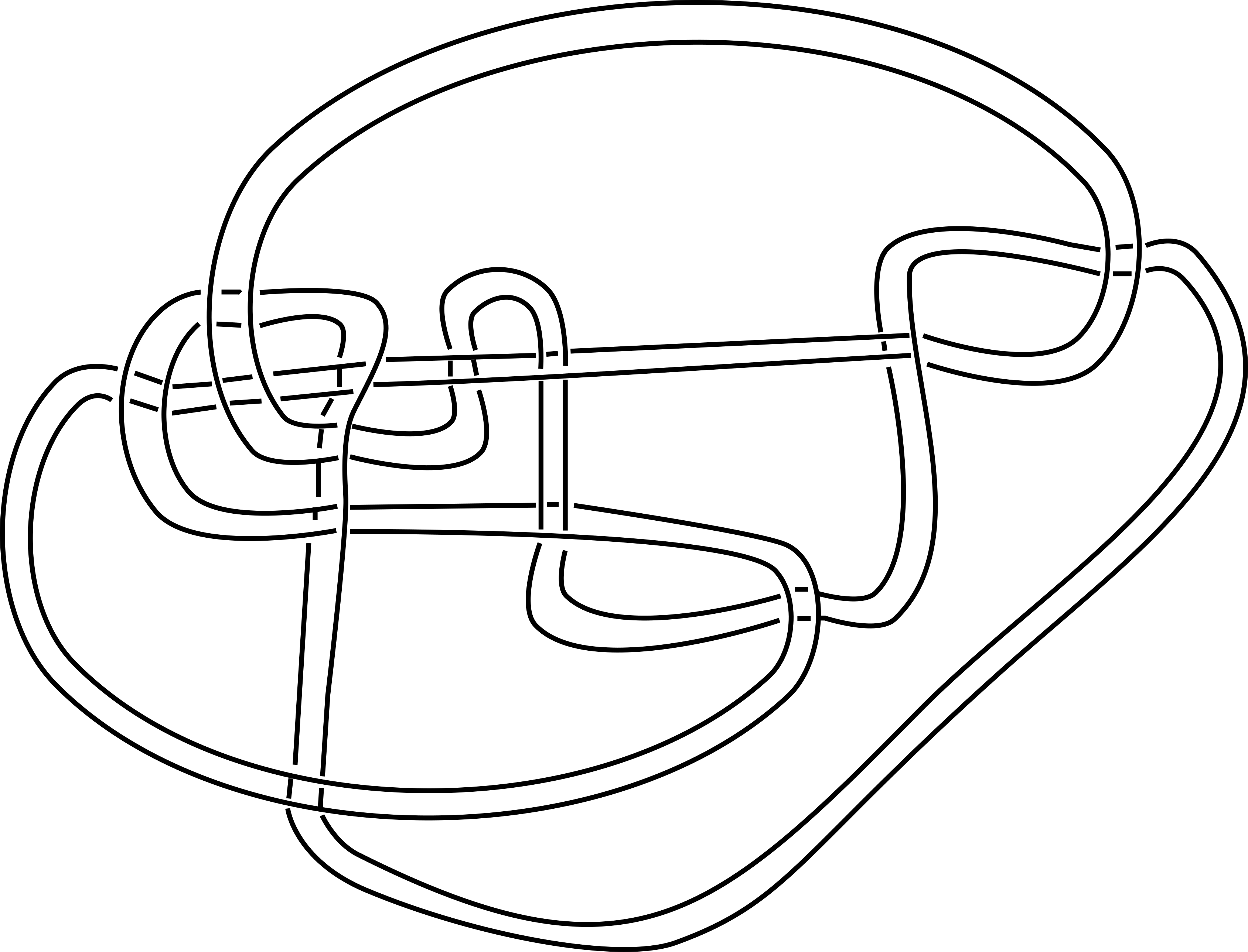}
        \caption{L13n2916.}
\end{figure}
\end{exmp}

\subsection*{Acknowledgments} I thank all reviewers for their thorough critiques. I am grateful to Maggy Tomova, and Ryan Blair for helpful discussions; and to Katie Burke for helping me create the table of upper bounds for the Wirtinger numbers of virtual knots.


\bibliographystyle{plain}

\end{document}